\documentclass[letterpaper,10pt]{article}

\usepackage{amsmath,wasysym,tikz,amssymb,amsthm}

\usepackage[spanish]{babel}

\usepackage{array}

\usepackage[bitstream-charter]{mathdesign}
\usepackage[T1]{fontenc}

\theoremstyle{plain}

\newtheorem{teorema}{Teorema}[section]
\newtheorem*{teoremaint}{Teorema principal}
\newtheorem{proposicion}[teorema]{Proposici\'on}
\newtheorem{lema}[teorema]{Lema}

\newtheorem{observacion}[teorema]{Observaci\'on}
\newtheorem{pregunta}{Pregunta}
\theoremstyle{definition}

\newtheorem{definicion}[teorema]{Definici\'on}

\makeatletter
\newcommand{\subjclass}[2][2020]{%
  \let\@oldtitle\@title%
  \gdef\@title{\@oldtitle\footnotetext{#1 \emph{Mathematics subject classification.} #2}}%
}
\newcommand{\keywords}[1]{%
  \let\@@oldtitle\@title%
  \gdef\@title{\@@oldtitle\footnotetext{\emph{Palabras clave.} #1.}}%
}
\makeatother

\topskip=\baselineskip  \headsep=0.3in \headheight=0in
\title{Puntajes m\'aximos en el juego de domin\'o}

\author{
Carlos E. Parra$^{1}$, Eugenio Trucco$^{2}$\\ \\
\small{$^{1,2}$Instituto de Ciencias F\'isicas y Matem\'aticas,}\\
\small{Universidad Austral de Chile.}\\
\small{\texttt{$^{1}$carlos.parra@uach.cl, $^{2}$etrucco@uach.cl}}
}

\date{}



\begin{document}
\maketitle

\abstract{En este trabajo estudiamos los puntajes m\'aximos que se pueden obtener en una partida, por equipos, en el juego del domin\'o. Concretamente; nosotros mostramos que, si la partida termina porque el juego esta trancado, entonces el puntaje m\'aximo que se puede obtener bajo este supuesto es de 107.} 


\section{Introducci\'on}
El domin\'o es un juego de mesa que consta de 28 fichas rectangulares, con dos extremos (cuadrados) cada una, los que pueden contener un espacio en blanco o una cantidad de puntos, que van desde 1 hasta 6. M\'as a\'un; por cada espacio en blanco o cantidad de puntos, existen 7 fichas diferentes, donde uno de los extremos es el señalado.

Por otra parte, existen muchas formas de jugar el domin\'o que pueden variar dependiendo del n\'umero de jugadores, n\'umero de fichas de cada jugador, n\'umero de puntos para ganar el juego, entre otros. A groso modo; el juego de domin\'o, consiste de una o varias partidas, con el objetivo de conseguir, al menos, la cantidad de puntos necesarios para ganar el juego. Cada partida se inicia\footnote{Si es la primera partida del juego, se suele iniciar con la ficha que tiene 6 puntos en cada uno de sus extremos. En el resto de las partidas, los inicios se alternan consecutivamente entre todos los jugadores.} colocando una ficha en la mesa, y se prosigue consecutivamente los turnos de cada jugador, colocando fichas que permitan concatenar algunos de los extremos de la partida o pasando (turno sin jugar), en el caso que no tenga fichas para realizar lo anterior. Una partida termina cuando un jugador ha colocado todas sus fichas en la mesa o cuando no es posible que uno de los jugadores, coloque una ficha (est\'a configuraci\'on se le conoce como una {\bf tranca}). En cualquiera de los dos casos, se le asignan los puntos\footnote{Si la partida finaliza porque un jugador jug\'o su última ficha, entonces se asigna la suma de los puntos entre los jugadores del otro equipo (respectivamente, de los otros equipos). Si la partida finaliza en tranca, entonces se asigna los puntos del equipo con m\'as puntos (respectivamente, la suma de los puntos del resto de los equipos) o no se asignan puntos, en caso de que los equipos tengan la misma cantidad de puntos.} a el equipo\footnote{Si cada jugador juega individualmente, entenderemos que cada equipo est\'a formado por una sola persona.} del jugador que coloco su \'ultima ficha o al equipo que tiene menos puntos, seg\'un corresponda.

En este trabajo, nos enfocaremos en el juego del domin\'o en el que participan 4 personas, que se agrupan en parejas y se enfrentan entre s\'i (ver \cite{reglasdomino}). Adem\'as, asumiremos que un equipo es el ganador del juego (y, por tanto, el juego termina) cuando este suma, al menos, 100 puntos. Esta versi\'on del domin\'o es la que usualmente se juega en centroam\'erica y, para casi todas las personas que jugaron esta versi\'on, han necesitado m\'as de una partida para poder ganar el juego. Por tanto, surge de manera natural preguntarse lo siguiente:
\begin{pregunta}
	 ?`Es posible ganar el juego del domin\'o con una sola partida?
\end{pregunta}
En el a\~no 1965, se desarroll\'o una partida, conocida como la \emph{partida inmortal} (ver en la siguiente secci\'on), en la que un equipo gana una partida de domin\'o y obtiene 111 puntos (ver \cite{dominosinbarreras}). Por tanto, la respuesta a la pregunta anterior es positiva.

Ahora, nosotros nos planteamos la siguiente pregunta que tiene inter\'es dentro de las \emph{Matem\'aticas}: 
\begin{pregunta}
	?`Cu\'al es el puntaje m\'aximo que puede obtener un equipo en una partida de domin\'o?
\end{pregunta}

En la literatura matem\'atica encontramos el un trabajo de Reiss (ver \cite{contarpartidas}) en el que se cuenta el n\'umero de posibles partidas en el domin\'o, concretamente existen $7\,\,959\,\,229\,\,931\,\,520$ partidas.

A continuaci\'on, exhibimos nuestro resultado principal en esta direcci\'on, que est\'a enmarcado dentro de las partidas que terminan en una tranca.

\begin{teoremaint}
El mayor puntaje que puede obtener un equipo, en una partida de domin\'o que termina en una tranca, es de 107 puntos.
\end{teoremaint}

Por \'ultimo, nosotros en este trabajo no respondemos la pregunta 2 en su totalidad. Por tanto, dicha pregunta a\'un sigue sin respuesta. Sin embargo, sospechamos que la respuesta es de 111 puntos. 

\section{Notaci\'on}
En esta secci\'on, introduciremos las notaciones que nos permitir\'an dar una demostraci\'on al Teorema principal de este trabajo (ver Teorema~\ref{teo:107}). Para ello, empezaremos con la siguiente representaci\'on de las fichas del domin\'o.
\subsection{Las fichas}\label{sub-fichas}
Cada ficha del domin\'o, la representaremos por pares de la forma $[a,b]$, donde $a$ y $b$ son n\'umeros enteros del 0 al 6. Adem\'as, en esta representaci\'on asumiremos que la ficha $[a,b]$, coincide con la ficha $[b,a]$.  Por tanto; todas las fichas del domin\'o, las podemos visualizar a continuaci\'on: 
\[
\begin{matrix}
[0,6]&[1,6]&[2,6]&[3,6]&[4,6]&[5,6]&[6,6]\\
[0,5]&[1,5]&[2,5]&[3,5]&[4,5]&[5,5]&\\
[0,4]&[1,4]&[2,4]&[3,4]&[4,4]& &\\
[0,3]&[1,3]&[2,3]&[3,3]& & &\\
[0,2]&[1,2]&[2,2]& & & &\\
[0,1]&[1,1]& & & & &\\
[0,0]& & & & & &\\
\end{matrix}
\]

\subsection{Representaci\'on de una partida}
A continuaci\'on, utilizaremos la representaci\'on de las fichas del domin\'o, para ilustrar la partida inmortal del domin\'o (ver \cite{dominosinbarreras}). Concretamente, los jugadores al principio de la partida, tendr\'an en sus manos las siguientes fichas: 
\begin{enumerate}
	\item[] Jugador 1: $[1,1]$, $[4,0]$, $[3,1]$, $[0,5]$, $[0,0]$, $[0,2]$ y $[0,6]$.
	\item[] Jugador 2: $[5,5]$, $[6,5]$, $[5,4]$, $[5,3]$, $[6,3]$, $[4,3]$ y $[3,3]$.
	\item[] Jugador 3: $[1,0]$, $[0,3]$, $[5,1]$, $[2,5]$, $[2,3]$, $[2,2]$ y $[1,2]$.
	\item[] Jugador 4: $[1,4]$, $[2,4]$, $[4,4]$, $[4,6]$, $[2,6]$, $[1,6]$ y $[6,6]$.   
\end{enumerate}
Ahora, representaremos la partida inmortal a trav\'es de la siguiente tabla:  

\begin{center}
\begin{tabular}{|c|c|c|c|c|c|c|c|}
\hline &Turno 1&Turno 2&Turno 3&Turno 4&Turno 5&Turno 6&Turno 7 \\ \hline
Jugador 1&[1,1]&[4,0]&[3,1]&[0,5]&[0,0]&[0,2]&[0,6]  \\ \hline
Jugador 2&pasa &pasa&pasa&pasa&pasa&pasa&\\ \hline
Jugador 3&[1,0]&[0,3]&[5,1]&pasa&pasa&[2,5]&\\ \hline
Jugador 4&[1,4]&pasa&pasa&pasa&pasa&pasa&\\ \hline
\end{tabular}
\end{center}

Para entender la partida que representa la tabla, debemos considerar lo siguiente:
\begin{enumerate}
	\item Los jugadores 1 y 3 forman un equipo, mientras que los jugadores 2 y 4 forman el otro equipo. 
	\item A partir de la segunda columna, se ilustra cada una de las jugadas que realiza cada jugador en la partida. De hecho, el jugador 1 inicia la partida con la ficha $[1,1]$. 
	\item El jugador 2, en su primer turno pasa.
	\item El jugador 3, en su primer turno juega la ficha $[1,0]$. Por tanto, la mesa de la partida queda de la siguiente manera: $$[1,1][1,0]$$
	\item El jugador 4, puede jugar en su primer turno una ficha que tenga un 0 o un 1. Él decide jugar la ficha $[1,4]$ y as\'i, la mesa de la partida queda de la siguiente forma: $$[4,1][1,1][1,0]$$
	\item El jugador 1, juega ahora su segunda ficha en su segundo turno y, el resto de los turnos de cada jugador, sigue este patr\'on hasta que el jugador 1 termina la partida, jugando su \'ultima ficha.  
\end{enumerate}

La tabla en s\'i, representa lo que ocurre al nivel de la mesa en la partida de domin\'o. En este caso, al finalizar la partida el jugador 2 se quedo con las fichas $[5,5]$, $[6,5]$, $[5,4]$, $[5,3]$, $[6,3]$, $[4,3]$ y $[3,3]$. Mientras que el jugador 3, se quedo con las fichas $[2,3]$, $[2,2]$ y $[1,2]$, y el jugador 4 con las fichas $[2,4]$, $[4,4]$, $[4,6]$, $[2,6]$, $[1,6]$ y $[6,6]$. Si se suman los puntos de las fichas que tienen los jugadores 2 y 4 al finalizar la partida, se obtiene 111, lo que le da la victoria de la partida y el juego al equipo formado por los jugadores 1 y 3.

\subsection{Partidas que finalizan en una tranca}
En esta subsecci\'on, introduciremos algunas nociones asociadas a las partidas que finalicen en una tranca. Dichas nociones son claves para la demostraci\'on del Teorema~\ref{teo:107}. Empezaremos resaltando el siguiente hecho
\begin{proposicion}\label{prop:trancaiguales}
Si una partida termina en una tranca, los extremos del tablero coinciden.
\end{proposicion}
\begin{proof}
Si la partida esta trancada, supongamos sin p\'erdida de generalidad que los extremos del tablero son 0 y 6. Adem\'as, supongamos que la ficha $[0,k]$ con $k\ne 0$ est\'a en un extremo del tablero. Ahora, usando el hecho de que las fichas de la partida van concatenadas, deducimos que hay un n\'umero par de 0 en la mesa, si quitamos la ficha $[0,k]$. Como no hay mas 0 en el resto de las fichas de los jugadores, se deduce entonces que en las fichas del domin\'o, hay un n\'umero impar de 0 (pues al n\'umero par de 0 mencionados anteriormente, le debemos sumar un 0 de la ficha $[0,k]$). Lo cual es absurdo, pues en total hay 8 (n\'umero par). 

\vspace{0,3 cm}

Ahora, debemos estudiar el caso en que la ficha $[0,0]$ est\'a en el extremo respectivo. En este caso, esta ficha esta conectada a una de la forma $[0,k]$ con $k\ne 0.$ Por un argumento similar al anterior, se deduce que en la mesa hay un n\'umero par de 0, si quitamos las fichas $[0,0]$ y $[k,0]$. Por ende, tambi\'en llegaremos a una contradicci\'on.
\end{proof}

En virtud de lo anterior, para una partida que finaliza en una tranca, diremos que dicha partida est\'a {\bf trancada por el n\'umero $\mathbf{k}$},  si este es el n\'umero que aparece en los extremos del tablero al finalizar la partida.

\begin{proposicion}
Las siguientes se cumplen, en una partida que finaliza en una tranca:
\begin{enumerate}
	\item[a)] 10 es la menor cantidad de fichas que aparece en el tablero, al finalizar la partida;
	\item[b)] 42 es la menor cantidad de puntos que suma el tablero, al finalizar la partida.
\end{enumerate}
\end{proposicion}
\begin{proof}
Si la partida est\'a trancada por el n\'umero $k$, al menos, todas estas fichas
\[[k,0],\ [k,1],\ [k,2],\ [k,3],\ [k,4],\ [k,5]\textnormal{ y }[k,6]\]
fueron jugadas. Salvo en los extremos, cada uno de los siguientes n\'umeros 0, 1, 2, 3, 4, 5 y 6, excluyendo el $k$, aparece una cantidad par de veces por la concatenaci\'on de fichas. Por ende, el m\'inimo de veces que aparece cada uno de estos n\'umeros es dos. Como tenemos seis n\'umeros que tienen que parecer, al menos una vez m\'as, esto puede ser realizado con un m\'inimo de tres fichas. Deducimos entonces que, es necesario tener al menos 10 fichas para realizar una tranca. El siguiente ejemplo, muestra una partida que finaliza en tranca en el que se utilizan 10 fichas. Por tanto, queda demostrado el item a).

\vspace{0,3 cm}
Los jugadores tendr\'an las siguientes fichas:
\begin{enumerate}
	\item[] Jugador 1: $[0,0]$, $[0,3]$, $[5,6]$, $[6,6]$, $[5,5]$, $[4,4]$ y $[3,3]$;
	\item[] Jugador 2: $[0,1]$, $[3,4]$, $[6,0]$, $[2,2]$, $[1,1]$, $[4,1]$ y $[6,4]$;
	\item[] Jugador 3: $[1,2]$, $[4,0]$, $[6,3]$, $[5,4]$, $[6,2]$, $[3,2]$ y $[5,3]$;
	\item[] Jugador 4: $[2,0]$, $[0,5]$, $[6,1]$, $[2,5]$, $[5,1]$, $[3,1]$ y $[4,2]$
\end{enumerate}
La partida en cuesti\'on, se representa a trav\'es de la siguiente tabla
\begin{center}
\begin{tabular}{|c|c|c|c|}
\hline  &Turno 1& Turno 2& Turno 3 \\ \hline
Jugador 1&[0,0]&[0,3]&[5,6]\\ \hline
Jugador 2&[0,1]&[3,4]&[6,0]\\ \hline
Jugador 3&[1,2]&[4,0]&\\ \hline
Jugador 4&[2,0]&[0,5]&\\ \hline
\end{tabular}
\end{center}
Para demostrar el item b), note que en la partida anterior los puntos que aparecen en el tablero al finalizar la partida son 42. Por otra parte, para una partida que este trancada por el n\'umero $k$, se necesitan jugar las 7 fichas que tienen el n\'umero $k$, junto con al menos tres fichas, de modo que, en el tablero aparezca cada uno de los siguientes n\'umeros, una cantidad par de veces: 1, 2, 3, 4 , 5 y 6, excluyendo el $k$. Por tanto, en el tablero al finalizar dicha partida tendremos al menos 8 veces el n\'umero k y dos veces cada uno de los siguientes n\'umeros: 1, 2, 3, 4 , 5 y 6, excluyendo el $k$. As\'i la suma de los puntos, es de al menos: 
$$8k+2(1+2+3+4+5+6-k)=6k+42$$
Con $k=0$ y 10 fichas, se obtiene el m\'inimo valor.    
\end{proof}

\begin{definicion}
	Una partida que finalice en una tranca, se dice que es una {\bf tranca m\'inima}, si la suma de los puntos del tablero al finalizar la partida es de 42.
\end{definicion}
En virtud de lo anterior, una tranca m\'inima es un tranca por 0 en el que se jugaron 10 fichas. Llamaremos {\bf conectores} a cada una de las tres fichas sin 0 que se juegan en dicha partida. De hecho, si $[a,b]$ es un conector entonces en el tablero al finalizar la partida aparecer\'a de la siguiente forma
$$\cdots[0,a][a,b][b,0]\cdots,$$
M\'as a\'un, no hay otros $a$ y $b$ en el tablero. En virtud de este hecho, finalizamos la subsecci\'on con la siguiente observaci\'on
\begin{observacion}\label{obs:nolleva}
En una tranca m\'inima, si un jugador pasa por el n\'umero $k\neq 0$, entonces dicho jugador no tenia fichas con el n\'umero $k$, desde el inicio de la partida (es decir, desde la mano inicial).
\end{observacion}
\section{Resultados}
En est\'a secci\'on; presentaremos nuestros resultados, en el que se incluye una demostraci\'on del teorema principal de este trabajo. Para ello, empezaremos mostrando una situaci\'on que ocurre en algunas de las trancas m\'inimas, con respecto al equipo perdedor de la partida.

\begin{proposicion}\label{prop:nopasanporelmismo}
Consideremos una tranca m\'inima. Si ambos jugadores del equipo perdedor no tienen fichas con 0 y 1 (respectivamente, 2, 3, 4, 5 y 6) en sus manos iniciales, entonces uno de ellos juega al menos una ficha.
\end{proposicion}
\begin{proof}
Supongamos que el equipo perdedor, no tiene en sus manos iniciales fichas con 0 y 1 (el resto de los casos, el argumento es similar). Entonces, el equipo ganador tiene inicialmente, las siguientes 14 fichas
\[[0,0],\ [0,1],\ [0,2],\ [0,3],\ [0,4],\ [0,5],\ [0,6]\]
\[[1,1],\ [1,2],\ [1,3],\ [1,4],\ [1,5],\ [1,6],\ [a,b]\]
para algunos $a$ y $b$,  n\'umeros enteros del 2 al 6. Notemos que entre estas fichas, a lo m\'as, hay dos conectores de la partida en cuesti\'on. Usando el hecho de que en dicha partida se juegan tres conectores, se deduce que el equipo perdedor tiene que jugar al menos una ficha (un conector).
\end{proof}

El siguiente lema nos indica la cantidad de puntos que tiene el juego del domin\'o en todas sus fichas.  Dicho dato es clave, para deducir los puntajes m\'aximos de ciertas partidas. 
\begin{lema}
La suma de los puntos presentes en todas las fichas del juego de domin\'o es 168.
\end{lema}
\begin{proof}
Cada n\'umero entero desde el 0 al 6, aparece 8 veces en las fichas (ver Subsecci\'on \ref{sub-fichas}). As\'i, el total de los puntos en las fichas es $$8 \cdot 0+8 \cdot 1+8 \cdot 2+8 \cdot3+8 \cdot4+8 \cdot5+8 \cdot6= 8\cdot(1+2+3+4+5+6)=8\cdot21=168.$$
\end{proof}
La siguiente proposici\'on, nos dice que no es posible que un jugador, al iniciar una partida de domin\'o, no tenga cuatro n\'umeros dados en su mano inicial. 
\begin{proposicion}\label{prop:sin4noexiste}
Dados cuatro n\'umeros enteros distintos entre 0 y 6, no existen siete fichas de domin\'o sin estos n\'umeros.
\end{proposicion}
\begin{proof}
Supongamos que $a$, $b$ y $c$ son los tres n\'umeros enteros del 0 al 6, que son distintos a los cuatro n\'umeros dados. Note que, las \'unicas fichas del domin\'o que contienen exclusivamente estos tres n\'umeros, son las siguientes:
$$[a,a],\ [a,b],\ [a,c],\ [b,b],\ [b,c],\ [c,c]$$
Por tanto, no existen 7 fichas de domin\'o sin incluir algunos de los cuatro n\'umeros dados. 
\end{proof}
Ahora, continuamos con el siguiente resultado:

\begin{lema}\label{lama:pasapor2}
Si en una tranca m\'inima, en un turno de un jugador del equipo perdedor este pasa por dos n\'umeros diferentes de 0, entonces el equipo perdedor jug\'o, al menos, una ficha.
\end{lema}
\begin{proof}
Recordemos que en una partida que termine con una tranca m\'inima, se juegan las siete fichas con 0 junto con tres conectores $[j_1,j_2],\ [j_3,j_4]$ y $[j_5,j_6]$, donde
\[\{j_1,j_2,j_3,j_4,j_5,j_6\}=\{1,2,3,4,5,6\},\]
En particular, todos los $j_r$'s son diferentes entre si. Ahora, supongamos que el equipo perdedor no juega una ficha en la partida en cuesti\'on. Bajo este supuesto, es claro que el equipo ganador inicia la partida. Adem\'as, supondremos que el jugador 1 (respectivamente, jugador 3) deja la mesa de la partida, despu\'es de jugar una ficha, de la siguiente forma: 
\[[j_r,\dots,j_s]\]
para algunos $r$ y $s$, distintos entre si, enteros del 1 al 6. Por tanto, el jugador 2 (respectivamente, jugador 4) pasa, al menos, por los n\'umeros $0$, $j_r$ y $j_s$. Consecutivamente, el jugador 3 (respectivamente, el jugador 1) podría pasar o podría jugar una ficha. En cualesquiera de los dos casos, despu\'es de dicho turno, en uno de los extremos del tablero aparecer\'a el n\'umero $j_s$ o el n\'umero $j_r$, seg\'un corresponda. Esto implicar\'ia que, el jugador 4 (respectivamente, jugador 2) tambi\'en pasar\'ia por el n\'umero $j_s$ o el n\'umero $j_r$. Ahora, desde la Observaci\'on \ref{obs:nolleva} junto con la Proposici\'on \ref{prop:nopasanporelmismo}, obtenemos una contradicci\'on. Dicha contradicci\'on viene de suponer que el equipo perdedor no juega una ficha en la partida señalada. Por tanto, deducimos que el equipo perdedor juega una ficha.         
\end{proof}

\begin{observacion}\label{obs:tableros}
Ahora, estamos en condiciones de decir un poco m\'as de lo que menciona el lema anterior (ver la siguiente proposici\'on). Para ello, es necesario resaltar que en una tranca m\'inima, la configuraci\'on del tablero al finalizar tal partida, tiene una de las siguientes formas $($esencialmente se diferencian por la posici\'on en que aparece la ficha $[0,0]$$)$:
\begin{enumerate}
\item[$\mathbf{T1}:=$] $[0,0][0,j_1][j_1,j_2][j_2,0][0,j_3][j_3,j_4][j_4,0][0,j_5][j_5,j_6][j_6,0]$, donde la ficha $[0,0]$ no fue la \'ultima en jugarse. 

\item[$\mathbf{T2}:=$] $[0,j_1][j_1,j_2][j_2,0][0,0][0,j_3][j_3,j_4][j_4,0][0,j_5][j_5,j_6][j_6,0].$

\item[$\mathbf{T3}:=$] $[0,j_1][j_1,j_2][j_2,0][0,j_3][j_3,j_4][j_4,0][0,0][0,j_5][j_5,j_6][j_6,0].$

\item[$\mathbf{T4}:=$] $[0,j_1][j_1,j_2][j_2,0][0,j_3][j_3,j_4][j_4,0][0,j_5][j_5,j_6][j_6,0][0,0]$, donde la ficha $[0,0]$ es la \'ultima en jugarse. 
\end{enumerate}
donde $j_1,\dots,j_6$ son los n\'umeros enteros del 1 al 6. Notemos que las siguientes afirmaciones son verdaderas: 
\begin{enumerate}
 \item[(a)] El equipo perdedor, pasa exclusivamente por 0 a lo m\'as tres veces, siempre que el equipo ganador sume 9 turnos en la partida $($esto pudiese ocurrir en los tableros $\mathbf{T1}$, $\mathbf{T2}$ y $\mathbf{T3}$$)$.

\item[(b)] El equipo perdedor pasa por 0 cuatro veces, siempre que el equipo ganador sume 10 turnos en la partida $($esto pudiese ocurrir, exclusivamente en el tablero $\mathbf{T4}$$)$. 
\end{enumerate}
	
\end{observacion}

\begin{proposicion}\label{prop:perd-jugo-1}
En una tranca m\'inima, el equipo perdedor jug\'o, al menos, una ficha.
\end{proposicion}
\begin{proof}

Supongamos que el equipo perdedor no juega ninguna ficha y, por tanto, el equipo que inicia la partida es el ganador. Como la tranca m\'inima tiene 10 fichas y los jugadores $2$ y $4$ siempre pasan, en cada uno de sus turnos, la partida se extiende al menos hasta el turno 5 del jugador 3, como muestra la siguiente tabla:

\begin{center}
\begin{tabular}{|c|c|c|c|c|c|}
\hline & Turno 1&Turno 2& Turno 3& Turno 4& Turno 5 \\ \hline
Jugador 1& & & & &    \\ \hline
Jugador 2& pasa & pasa & pasa & pasa & pasa  \\ \hline
Jugador 3& & & & &  \\ \hline
Jugador 4& pasa & pasa & pasa & pasa &  \\ \hline
\end{tabular}
\end{center}

En virtud del Lema~\ref{lama:pasapor2}, cada uno de los pases que muestra la tabla es exclusivamente por 0 o por 0 y otro n\'umero $j\neq 0$. Desde la observaci\'on anterior, podemos deducir que existen al menos 5 pases del equipo perdedor, donde el jugador en cuesti\'on no lleva $0$ y $j$, y todos los $j'$s son distintos entre si. Como ambos jugadores pasan por $0$, entonces deducimos que existe un jugador del equipo perdedor que pasa por $0$ y tres n\'umeros, distintos entre si, no nulos. La Observaci\'on \ref{obs:nolleva} junto con la Proposici\'on \ref{prop:sin4noexiste}, nos dicen que tenemos una contradicci\'on. Por tanto, el equipo perdedor juega en la partida, al menos una ficha.
\end{proof}

\begin{observacion}
Los argumentos usados en la demostraci\'on de la proposici\'on anterior, tambi\'en son validos para partidas que terminen en una traca por un n\'umero distinto de 0 y en el que se utilicen (un total de) 10 fichas. 
\end{observacion}

Notemos que en \cite{dominocompetitivo} se afirma que el mayor puntaje posible de obtener en una tranca es 126, sin embargo, un análisis más detallado muestra lo siguiente:

\begin{observacion}\label{obs:key-1}
Otra de las cosas que podemos deducir de la proposici\'on anterior es que, el equipo perdedor se puede quedar con a lo m\'as trece fichas al finalizar una partida que termine en una tranca. En una primera aproximaci\'on, el equipo perdedor podr\'ia quedarse con las trece fichas de mayor puntaje del domin\'o. Dichas fichas, incluir\'ian las siguientes doce: 
\[[6,6],\ [5,6],\ [4,6],\ [5,5],\ [3,6],\ [4,5],\ [2,6],\ [3,5],\ [4,4],\ [1,6],[2,5],\ [3,4]\]
y para completar la ficha trece, podemos tomar alguna de las siguientes: 
\[[0,6], \ [1,5], \ [2,4], \ [3,3]\]
La suma de los puntos, de las trece fichas de mayor puntaje es 112. Lo que puede hacer pensar que, en una tranca m\'inima este es el puntaje m\'aximo posible a obtener. Sin embargo, si el equipo perdedor se queda con a lo m\'as trece fichas, el equipo ganador se queda con, al menos, cinco fichas. Las cinco fichas, sin el 0, con menor cantidad de puntos son las siguientes:
\[[1,1], \ [1,2],\ [1,3], \ [2,2]\]
y una de entre $[1,4]$ y $[2,3].$ La suma de los puntos en estas cinco fichas es 18. Por lo tanto, en una tranca m\'inima, la suma de los puntos en las fichas del equipo perdedor, a lo m\'as son:
\[\text{(Puntos totales)-(puntos del tablero)-(puntos del equipo ganador al finalizar la partida)=}168-42-18=108\]

Siguiendo el an\'alisis anterior, pero, para trancas donde se utilicen 10 fichas, podemos hacer la siguiente estimaci\'on:

\begin{center}
\begin{tabular}{|c|c|c|c|}
\hline Tranca por & Mesa & 5 fichas menores & Cota m\'axima Posible \\ \hline
0&42&$[1,1],[1,2],[1,3],[2,2],[2,3]$& $108=168 -42-18$\\  \hline
1&48&$[0,0],[0,2],[0,3],[2,2],[0,4]$& $107=168-48-13$\\ \hline
2&54&$[0,0],[0,1],[1,1],[0,3],[1,3]$& $104=168-54-10$\\ \hline
3&60&$[0,0],[0,1],[1,1],[0,2],[1,2]$& $98=168-60-10$\\ \hline
4&66&$[0,0],[0,1],[1,1],[0,2],[1,2]$& $93=168-66-9$\\ \hline
5&72&$[0,0],[0,1],[1,1],[0,2],[1,2]$& $87=168-72-9$\\ \hline
6&78&$[0,0],[0,1],[1,1],[0,2],[1,2]$& $81=168-78-9$\\ \hline
\end{tabular}
\end{center}
\end{observacion}
Continuamos con el siguiente resultado, donde se puede explicitar una partida donde se obtienen 107 puntos. Este ser\'ia un ejemplo, distinto a la partida inmortal, que tambi\'en da una respuesta positiva a la Pregunta 1.
\begin{proposicion}\label{prop:exam}
Existe una tranca m\'inima, donde el equipo ganador obtiene 107 puntos.
\end{proposicion}
\begin{proof}
Supongamos que al principio de la partida, los jugadores en cuesti\'on, tienen las siguientes fichas:
\begin{enumerate}
	\item [] Jugador 1: $[0,0], \ [1,4], \ [0,5], \ [3,0],\ [2,0],\ [1,1]$ y $[1,2];$
	\item [] Jugador 2: $[1,5],\ [1,6],\ [2,5],\ [2,6],\ [5,5],\ [5,6]$ y $[6,6];$
	\item [] Jugador 3: $[0,1],\ [0,4],\ [0,6],\ [2,3], \ [1,3],\ [2,2]$ y $[2,4];$
	\item [] Jugador 4: $[3,3], \ [3,4],\ [3,5],\ [3,6],\ [4,4],\ [4,5]$ y $[4,6].$
\end{enumerate}

Ahora, consideremos la partida que se desarrolla siguiendo lo indicado en la tabla:

\begin{center}
\begin{tabular}{|c|c|c|c|c|c|}
\hline & Turno 1&Turno 2& Turno 3& Turno 4& Turno 5 \\ \hline
Jugador 1&[0,0] &[1,4] &[0,5] &[0,3] & [0,2] y se tranca   \\ \hline
Jugador 2& pasa 0 & pasa 0 y 4 & [5,6] & pasa 0 y 3 & \\ \hline
Jugador 3& [0,1] & [0,4] & [6,0] & [3,2] &  \\ \hline
Jugador 4& pasa 0 y 1 & pasa 0 & pasa 0 & pasa 0 y 2 &  \\ \hline
\end{tabular}
\end{center}

En este caso, el equipo perdedor se queda con las siguientes fichas al finalizar la partida: 
\[[1,5],\ [1,6],\ [2,5], \ [2,6], \ [5,5],\ [6,6],\ [3,3],\ [3,4],\ [3,5],\ [3,6],\ [4,4],\ [4,5]\textnormal{ y }[4,6]\]
las que suman 107 puntos.
\end{proof}
Desde la Observaci\'on \ref{obs:key-1}, sabemos que el puntaje m\'aximo posible a obtener, para el equipo ganador de una tranca m\'inima es de 108. Ahora, estamos en condiciones de enunciar nuestro resultado principal (en el que se descarta est\'a opci\'on).
\begin{teorema}\label{teo:107}
El mayor puntaje que puede obtener un equipo, en una partida de domin\'o que termina en una tranca, es de 107 puntos. 
\end{teorema}
Desde el an\'alisis que hicimos para obtener la tabla en la Observaci\'on \ref{obs:key-1}, es claro que, la partida que finaliza en una tranca con mayor puntaje posible (entre todas las posibles) para el equipo ganador, es una tranca m\'inima. Por tanto, desde la Proposici\'on \ref{prop:exam} junto con el siguiente resultado, se obtiene una prueba del Teorema \ref{teo:107}.  

\begin{proposicion}\label{prop:108}
No existe una tranca m\'inima, en la cual el equipo ganador obtenga 108 puntos.
\end{proposicion}
\begin{proof}
	La demostraci\'on la realizaremos por el absurdo. Supongamos que dicha partida existe y llegaremos a una contradicci\'on. En efecto, si tal partida existiera, entonces las siguientes afirmaciones se cumplen:
	\begin{enumerate}
		\item[(a)] El equipo perdedor juega una \'unica ficha en la partida.
		\item[(b)] La ficha $[5,6]$ no se juega en dicha partida. 
	\end{enumerate}
	Desde Proposici\'on \ref{prop:perd-jugo-1} junto con el hecho de que, las 12 fichas con m\'as puntos del domin\'o, suman 106 puntos, obtenemos que el item (a) es verdadero. Por otra parte, el equipo ganador al final de la partida, tendr\'a en sus manos las siguientes 5 fichas (divididos en dos casos), que son las de menor puntaje, dentro de las 18 fichas que quedan en las manos de todos los jugadores al finalizar la partida en cuesti\'on (ver Observaci\'on \ref{obs:key-1}):
	\begin{enumerate}
		\item [] {\bf Caso 1:} $[1,1], \ [1,2], \ [1,3], \ [2,2] $ y $[2,3]$
		\item [] {\bf Caso 2:} $[1,1], \ [1,2], \ [1,3] , \ [2,2] $ y $[1,4]$
	\end{enumerate}
	Ahora, si la ficha $[5,6]$ estuviese en la partida, entonces los otros dos conectores de la partida ser\'ian: $[1,2]$ y $[3,4]$ \'o $[1,3]$ y $[2,4]$ \'o $[1,4]$ y $[2,3]$. Pero, ninguno de estos casos son posibles, dado que uno de estos conectores, est\'a en las manos de un jugador del equipo ganador al finalizar la partida (ver casos 1 y 2, señalados previamente). Por ende, el item (b) se cumple. 
	
	Por otra parte, desde el item (a) obtenemos que el equipo perdedor suma, al menos, un total de 7 pases en toda la partida. Por un argumento similar al señalado en la Observaci\'on \ref{obs:tableros} junto con la demostraci\'on del Lema \ref{lama:pasapor2}, el equipo perdedor tiene, al menos, cuatro pases que no son exclusivamente por el 0. Por tanto, existen $j_1, \ j_2, \ j_3$ y $j_4$ n\'umeros, distintos entre s\'i, enteros del 1 al 6, tales que alg\'un jugador del equipo perdedor pasa por $j_1$ $($respectivamente, $j_2$, $j_3$ y $j_4$$)$ y, por tanto, dicho jugador no ten\'ia fichas con el n\'umero $j_1$ $($respectivamente, $j_2$, $j_3$ y $j_4$$)$, al principio de la partida $($ver Observaci\'on \ref{obs:nolleva}$)$.

Para continuar con la demostraci\'on, notemos que desde el item (b) sabemos que los conectores de la partida en cuesti\'on, son de la siguiente forma:
$$[5,a], \ [6,b], \ [c,d]$$
donde $\{a,b,c,d\}=\{1,2,3,4\}$. A continuaci\'on, mostraremos lo siguiente: 

\vspace{0,3 cm}
{\bf Afirmaci\'on: }Ning\'un jugador del equipo perdedor, pasa por 5 $($respectivamente, 6$)$

En efecto, si este fuese el caso, entonces el otro jugador del equipo perdedor, tiene en sus manos, al finalizar la partida, las siguientes fichas: $$[5,b], \ [5,c], \ [5,d], \ [5,6], \ [5,5]$$
Por tanto, dicho jugador pasa por el valor $a$. Pues, en caso contrario, el jugador que pasa por 5, pasar\'ia por los 4 n\'umeros distintos $j_1$, $j_2$, $j_3$ y $j_4$. Lo cual es absurdo, en virtud de la Proposici\'on \ref{prop:sin4noexiste}. M\'as a\'un, el jugador que pasa por 5 (del equipo perdedor), tiene una ficha con 0 al iniciar la partida, dado que en caso contrario, estaríamos en las condiciones de la Proposici\'on \ref{prop:sin4noexiste} y, por tanto, tendremos una contradicci\'on. Forzosamente, en este supuesto, la única ficha que juega el equipo perdedor es la ficha que tiene un 0, que tiene el jugador (del equipo perdedor) que pasa por 5.   

\vspace{0,3 cm}

Por otra parte, notemos que la ficha $[6,a]$ no se juega y, por los casos 1 y 2, junto con lo anterior, se deduce que dicha ficha la tiene el jugador del equipo perdedor que pas\'o por 5. Entonces, cuando en la partida se jueguen los conectores $[5,a]$ y $[c,d]$ el equipo perdedor, podrá jugar una ficha m\'as (que no contenga 0), lo que contradice el item (a).  
	
	\vspace{0,3 cm}
Entonces, nuestra afirmaci\'on es verdadera. Dicha afirmaci\'on, nos dice que los jugadores del equipo perdedor tienen fichas con el n\'umero 5 y con el n\'umero 6. Por tanto, cuando en la partida se juegue los conectores $[5,a]$ y $[6,b]$ el equipo perdedor podr\'a jugar, al menos, dos fichas en la partida en cuesti\'on. Pero, esto contradice el item (a). Como queríamos.
\end{proof}


\end{document}